\documentclass{article}

\usepackage{graphicx} 
\usepackage{amsmath}
\usepackage{amsthm}
\usepackage{amssymb}
\usepackage[margin=1in]{geometry}
\usepackage{hyperref,color}
\hypersetup{colorlinks=true,citecolor=blue, linkcolor=blue, urlcolor=blue}{\tiny }

\theoremstyle{plain}
\newtheorem{thm}{Theorem}[section]
\newtheorem{theorem}[thm]{Theorem}

\newtheorem{lemma}[thm]{Lemma}

\newtheorem{fact}[thm]{Fact}
\newtheorem{remark}[thm]{Remark}

\theoremstyle{definition}

\newtheorem*{ass*}{Assumption}
\newtheorem*{assumption*}{Assumption}

\newcommand{\E}{\mathrm{E}}
\newcommand{\Var}{\mathrm{Var}}
\newcommand{\R}{\mathbb{R}}

\def\bs{\beta_{\textsc{l}}}
\def\bS{\beta_{\textsc{r}}}
\def\bc{\beta_{\textnormal{\textsc{cr}}}}

\begin{document}

\title{Cutoff for the Swendsen–Wang dynamics on the complete graph
}

 \author{
 	Antonio Blanca\thanks{Department of Computer Science and Engineering, Pennsylvania State University. Email: ablanca@cse.psu.edu. Research supported in part by NSF CAREER grant 2143762.}
 	\and
 	Zhezheng Song\thanks{Department of Computer Science and Engineering, Pennsylvania State University. Email: zzsong@psu.edu.}  	 	
 }
\date{\today}
            
\maketitle

\thispagestyle{empty}

\begin{abstract}
We study the speed of convergence of the Swendsen--Wang (SW) dynamics for the $q$-state ferromagnetic Potts model on the $n$-vertex complete graph, known as the mean-field model. The SW dynamics was introduced as an attractive alternative to the local Glauber dynamics, often offering faster convergence rates to stationarity in a variety of settings. A series of works have characterized the asymptotic behavior of the speed of convergence of the mean-field SW dynamics for all $q \ge 2$ and all values of the inverse temperature parameter $\beta > 0$. In particular,
it is known that when $\beta > q$ the mixing time of the SW dynamics is $\Theta(\log n)$. 
We strengthen this result by showing that for all $\beta > q$, there exists a constant $c(\beta,q) > 0$ such that the mixing time of the SW dynamics is $c(\beta,q) \log n + \Theta(1)$.
This implies that the mean-field SW dynamics exhibits the cutoff phenomenon in this temperature regime, demonstrating that this Markov chain undergoes a sharp transition from ``far from stationarity'' to ``well-mixed'' within a narrow $\Theta(1)$ time window. 
The presence of cutoff is algorithmically significant, as simulating the chain for fewer steps than its mixing time could lead to highly biased samples.

\end{abstract}

\vfill

\pagebreak

\setcounter{page}{1}

\section{Introduction}

The $q$-state ferromagnetic Potts model is a classical spin system model central in statistical physics, theoretical computer science, and discrete probability. 
Given a finite graph $G = (V,E)$, a set of spins or colors $\mathcal S = \{1,\dots,q\}$, and an edge interaction parameter $\beta > 0$, it is defined as a probability distribution over the configurations in $\Omega = \{1,\dots,q\}^{V}$, which correspond to the spin assignments from $\mathcal S$ to the vertices of $G$.
Formally, the probability of a configuration $\sigma \in \Omega$ is given by:
\begin{equation}
\label{eq:potts}
    \mu(\sigma) = \frac{1}{Z} \cdot e^{\beta M(\sigma)},
\end{equation}
where $M(\sigma)$ denotes the number of edges of  $G$ with the same spin in both of its endpoints in $\sigma$ (i.e., the monochromatic edges of $G$ in $\sigma$), and $Z$ is the normalizing constant or partition function. 
The edge interaction parameter $\beta$ is proportional to the inverse temperature in physical applications; the special case $q = 2$ corresponds to the classical ferromagnetic Ising model.

The algorithmic problem of efficiently generating approximate samples from~
$\mu$ has attracted much attention in theoretical computer science. For the Ising model (i.e., for $q = 2$), seminal results from the 1990s~\cite{JSIsing,RW} provide a polynomial-time approximate sampling algorithm for $\mu$. Additional approximate-sampling algorithms 
have been discovered since then for the ferromagnetic Ising setting~\cite{Garoni,GuoJer},
but all have a large running time of at least $O(n^{10})$ where $n = |V|$. For $q \ge 3$, the sampling problem is computationally hard, specifically \#BIS-hard~\cite{GoldbergJerrum,GSVY}, but efficient approximate sampling algorithms are still possible for many interesting families of graphs; see~\cite{HPR-PTRF,JKP,Carlson,carlson2020efficient,BG23-FOCS-abstract}.

An algorithm with the potential of providing both faster algorithms for the $q=2$ case and a technology for understanding the feasibility boundary of approximate sampling for $q \ge 3$ is the \emph{Swendsen--Wang (SW) dynamics}~\cite{SW}.
This sophisticated Markov chain exploits the connection between the ferromagnetic Potts model and the random-cluster model~\cite{ES} and was designed to bypass some of the key difficulties associated with sampling from $\mu$ at low temperatures (large $\beta$). 
In particular, 
the SW dynamics
is conjectured to converge to $\mu$ in $O(n^{1/4})$ steps when $q=2$ on any graph $G$ for any $\beta > 0$. This Markov chain was also recently used to the expand the graph families for which approximately sampling can be efficiently done for $q \ge 3$ and $\beta$ large~\cite{BG23-FOCS-abstract}.

The SW dynamics transitions from a configuration $\sigma_t\in \Omega$ to $\sigma_{t+1}\in \Omega$ as follows:
\begin{enumerate}
    \item Independently, for every $e = \{u,v\}\in E$ if $\sigma_t(u) = \sigma_t(v)$ include $e$ in $M_t$ with probability~$p = 1 - e^{-\beta}$;
    \item Independently, for every connected component $\mathcal C$ in $(V,M_t)$, draw a spin $s \in \{1,\dots,q\}$ uniformly at random, and set $\sigma_{t+1}(v)= s$ for all $v\in \mathcal C$. 
\end{enumerate}
The SW dynamics is ergodic and reversible with respect to $\mu$ and thus converges to it~\cite{SW,ES}. We refer to Step 1 as the percolation step and to Step 2 as the coloring step. 

The SW dynamics is a highly non-local Markov chain, updating and potentially changing the spin of each vertex in each step, making its analysis quite difficult in most settings. Still, 
significant progress has been made in understanding the speed of convergence of this dynamics in various geometries, including the complete graph~\cite{CF,GoJe,LNNP,galanis2019swendsen,GLP},
finite subsets of the $d$-dimensional integer lattice graph ${\mathbb{Z}}^d$~\cite{Ullrich-random-cluster,BSz2,BCSV,NamSly,BCPSV},
trees~\cite{Huber,BZSV-SW-trees}, 
random graphs~\cite{BG20,PottsRGMetastabilityCMP,BG22},
locally-tree-like graphs~\cite{BG23-FOCS-abstract},
graphs with sub-exponential growth~\cite{BG23-FOCS-abstract},
graphs of bounded~\cite{BCV20,BCCPSV} or unbounded degree~\cite{BZ},
as well as general graphs when $q=2$~\cite{GuoJer}. A setting that has attracted particular interest and that will be the focus of this work is the complete graph, also known as the \emph{mean-field model}.

On the complete graph, it is natural to re-parametrize the model and use $\beta/n$ instead of $\beta$; i.e., $\mu (\sigma) \propto \exp(\frac{\beta}{n} \cdot M(\sigma))$.
A detailed connection between the phase transition of the mean-field model, which we describe in detail in Remark~\ref{rmk:pt}, and 
the \emph{mixing time}
of the SW dynamics emerged from~\cite{LNNP,galanis2019swendsen,GLP}. 
The {mixing time} of a Markov chain 
is defined as the number of steps until the Markov chain is close in total variation distance to
its stationary distribution starting from the worst possible initial configuration, and it is the most standard notion of speed of convergence to stationarity. Formally, if $P$ is the transition matrix of the chain,
for $\varepsilon \in (0,1)$, 
\[
\tau_{\rm mix}(\varepsilon) = \max_{\sigma \in \Omega }\min_{t \ge 0} \{ \|P^t(\sigma,\cdot) - \mu(\cdot)\|_{\textsc{tv}} \le \varepsilon  \},
\]
and by convention $\tau_{\rm mix} := \tau_{\rm mix}(1/4)$.

When $q \ge 3$, the mixing time $\tau_{\rm mix}^{\mathrm{SW}}$ of the SW dynamics satisfies:
\begin{equation}
\label{eq:mt}
	\tau_{\rm mix}^{\mathrm{SW}} = 
	\begin{cases}
	\Theta(1) 			& \textrm{if}~\beta  < \bs; \\
	
	\Theta(n^{1/3})    			& \textrm{if}~\beta = \bs; \\
	
	e^{\Omega({n})}    & \textrm{if}~\beta \in (\bs,\bS); \\

    \Theta(\log n) & \textrm{if}~\beta \ge \bS;
	\end{cases}	
\end{equation}
see~\cite{galanis2019swendsen,GLP}.
It is known that $\bS = q$, but $\bs$ does not have a closed form, and it is given instead implicitly as the root of a certain equation~\cite{galanis2019swendsen,LL}. 
When $q =2$, $\bs = \bS = q$ and there is no slow-mixing window; in particular,
the mixing time of the mean-field SW dynamics for $q=2$ is $\Theta(1)$ for $\beta < q$, $\Theta(n^{1/4})$ at the critical point $\beta = q$, and $\Theta(\log n)$ for $\beta > q$~\cite{LNNP}.

In this work, we provide sharper results for the mixing time of the mean-field SW dynamics in the $\beta > \bS=q$ regime, further refining our understanding of this dynamics. We show that there exists a constant $c(\beta,q) > 0$ such that $\tau_{\rm mix}^{\mathrm{SW}} = c(\beta,q) \log n + \Theta(1)$,
which implies that this Markov chain exhibits \emph{cutoff}. A Markov chain exhibits cutoff
if there is a sharp drop in the total variation distance from close to $1$ to close $0$ in an interval of time of smaller order than $\tau_{\rm mix}$ which is called the \emph{cutoff window}; i.e., we say that there is cutoff if as $n \to \infty$, for any fixed $\varepsilon \in (0,1)$,
\[
\tau_{\rm mix}(\varepsilon) - \tau_{\rm mix}(1-\varepsilon) = o(\tau_{\rm mix}).
\]
We can now state our main result.

\begin{theorem}
\label{thm:main:intro}
  Let $q \ge 2$ and $\beta > \bS=q$. Then, there exists a constant
  $c(\beta,q) > 0$
  such that SW dynamics for the 
  $q$-state mean-field ferromagnetic Potts model exhibits cutoff at mixing time
 $\tau_{\rm mix}^{\mathrm{SW}} = c(\beta,q) \log n$ with cutoff window $\Theta(1)$; that is, $\tau_{\rm mix}^{\mathrm{SW}} = c(\beta,q) \log n + \Theta(1)$.
\end{theorem}
\noindent
We provide the expression for the constant $c(\beta,q)$ in~\eqref{eq:const}. 

The only previously known mixing time cutoff result for the SW dynamics 
was on the integer lattice graph for sufficiently small enough $\beta$~\cite{NamSly}; ours is the first such result in the low-temperature SW dynamics.
For the local Glauber dynamics, which updates a single randomly chosen vertex in each step, cutoff is known in the mean-field through the high-temperature $\beta < \bs$ regime~\cite{levin2010glauber,cuff2012glauber}. For the special case of the Ising model,
cutoff has also been established for the Glauber dynamics for several graph families; see~\cite{lubetzky2013cutoff,lubetzky2017universality,kim2021cutoff,van2021glauber}.

Recall from~\eqref{eq:mt} that the mixing time of the SW dynamics is $\Theta(1)$ when $\beta < \bs$ and exponentially slow when $\beta \in (\bs,\bS)$, so the question of whether the chain exhibits cutoff is of limited significance in those parameter regimes. 
It is an interesting open question, however, whether the SW dynamics exhibits cutoff at the dynamical critical points $\beta=\bs$ or $\beta=\bS$. In the former, mixing is not entirely governed by a ``strong drift'' towards a typical spin count distribution and relies on the variance of the dynamics as well; as such, it is unlikely that the dynamics exhibits cutoff at $\beta = \bs$. For $\beta = \bS$, the SW dynamics exhibits a first-order drift towards typical configurations, but the percolation step is critical, which significantly complicates even the asymptotic mixing time analysis. 

From an algorithmic perspective, for a Markov chain like the SW dynamics 
that exhibits cutoff at $c(\beta,q) \log n$, knowing only the asymptotic mixing time, e.g., that $\tau_{\rm mix}^{\mathrm{SW}} = \Theta(\log n)$, is insufficient in practice, since simulating the dynamics for say $\frac{c(\beta,q)}{2} \log n$ steps, is guaranteed to result in a sample with total variation distance close to $1$.

We conclude this introduction with some brief remarks about our proof and techniques. We construct a multi-phase coupling that converges to a pair of configurations with roughly the correct (expected) number of vertices in each spin class. (In the mean-field 
setting
when $\beta > \bS$, under the stationary distribution $\mu$, there is one dominant spin and all other spins are assigned to roughly the same number of vertices.) 
Our coupling gradually contracts the distance to this type of configuration over two phases: first to a linear distance and then to within $O(\sqrt{n})$ distance. The second phase determines the leading order of the mixing time. Earlier analyses (i.e.,~\cite{galanis2019swendsen}) yield the  $O(\log n)$ mixing time bound but treat step deviations ``pessimistically'' and cannot provide the sharper bound we derive here. 
Our analysis carefully accounts for how deviations are amortized over time and account for the fact that early fluctuations do not matter as much. 

An innovation of the present argument is the use of a $q \times q$ projection to a partition $\{V_1,\dots,V_q\}$ of $V$ at \emph{low temperatures}. This projection contains 
information about the number of vertices assigned each spin in each $V_i$ and its convergence is known to imply that of the SW dynamics due to the symmetry of the mean-field model. This idea
was used before in~\cite{cuff2012glauber,LNNP,galanis2019swendsen} in the high-temperature $\beta < \bs$ setting where the percolation steps of the SW dynamics result in sub-critical random (sub)graphs with trivial connected component structures. In the $\beta > \bS$ regime, we must account for the existence of a linear-sized connected component and how it is distributed in the partition $\{V_1,\dots,V_q\}$. To address this, we have an additional step in our analysis in which we argue that the
fraction of vertices assigned each spin in each $V_i$ coalesces to the overall fraction of vertices with that spin in the full configuration before the full configuration reaches a typical spin count (i.e., before the end of second phase of the coupling). This additional guarantee allows us to design the required coupling.

\begin{remark}
    \label{rmk:pt}
As mentioned earlier, the mixing behavior of the SW dynamics described in~\eqref{eq:mt} is tightly connected to the order-disorder phase transition of the mean-field model which we describe next for completeness. This phase transition
occurs at the {critical} value $\beta = \bc(q)$, where $\bc(q)=q$ when $q = 2$
and 
$$
\bc(q)=2\Big(\frac{q-1}{q-2}\Big)\log(q-1)
$$ 
for $q \ge 3$. When $\beta < \bc(q)$, the number of vertices assigned each spin is roughly the same with high probability up to lower-order fluctuations. In contrast, when
$\beta>\bc(q)$ there is a dominant spin in the configuration with high probability.
For $q \ge 3$, the model exhibits \emph{phase coexistence} at the critical threshold $\beta = \bc(q)$; this means that at this point, the set of configurations with no dominant spin and the 
set of configurations with a dominant spin (with all other non-dominant spins assigned to roughly the same number of vertices), contribute each a constant fraction of the probability mass with all other configurations contributing 0 mass (in the limit as $n\rightarrow\infty$). The effects of the phase coexistence phenomenon extends to the window $(\bs,\bS)$ around $\bc(q)$; this window is known as the metastability window and coincides with slow-mixing regime for the SW dynamics. For $q = 2$, on the other hand, there is no phase coexistence at $\bc(q)$ or metastability, so there is no slow mixing.
\end{remark}

\section{Preliminaries}

In this section, we gather a number of standard definitions and results that we will use in our proofs.

\subsection{Couplings}

A one step coupling of the Markov chain 
with state space $\Omega$
specifies
for every pair of states $(X_t, Y_t)$ a probability distribution over $(X_{t+1}, Y_{t+1})$ such that the processes $\{X_t\}$ and $\{Y_t\}$ are each faithful realizations of the chain, and if $X_t=Y_t$ then $X_{t+1}=Y_{t+1}.$ The {\it coupling time} 
is defined by
\[T_{\mathrm{coup}} = \max_{x,y \in\Omega}\min\limits_t \{X_t=Y_t \mid X_0=x,Y_0=y\}.\]
It is a standard fact that if $\Pr(T_{\mathrm{coup}} > t) \le \varepsilon$, then
$\tau_{\rm mix}(\varepsilon) \le t$; see~\cite{LPW}.
In addition, to bound $\tau_{\rm mix} = \tau_{\rm mix}(1/4)$, it suffices to bound $\Pr(T_{\mathrm{coup}} > t)$ by any constant $\varepsilon < 1$, as 
the bound can then be boosted by repeating the coupling a constant number of times.
For establishing cutoff,
however, this approach is not feasible, as it would increase the mixing time bound by a multiplicative factor. Throughout our proofs, we track the probability of success of each phase of our multi-phase coupling so that $\Pr(T_{\mathrm{coup}} > t) \le 1/4$.

\subsection{Random graph lemmas}

Let $\mathcal G$ be a random graph distributed as a 
$\mathcal G(n,\lambda/n)$ random graph with $\lambda > 0$.
Let $L_i(\mathcal G)$ denote the size of the $i$-th largest component of $\mathcal G$ (breaking ties arbitrarily) and let $C(v)$ denote the connected component of vertex $v$ in $\mathcal{G}$.

\begin{lemma}[Lemma 5.7~\cite{LNNP}]
\label{lemma:rg:isolated}
Let $I$ be the number of isolated vertices in $\mathcal G$.
For any constant $\lambda > 0$, there exists a constant $C > 0$ that
$\Pr(I \ge C n) = 1 - O(n^{-1})$.
\end{lemma}

For $\lambda > 1$ let $\theta(\lambda)$ denote the unique positive solution of the equation
\begin{equation}
    \label{eq:rg:root}
   e^{-\lambda x} = 1-x. 
\end{equation}

\begin{lemma}[Lemma 5.4~\cite{LNNP}]
\label{lemma:rg:gc}
For $\lambda > 1$, there exist constants $c = c(\lambda) > 0$ and $C > 0$ such that for any $A>0$ 
$$
\Pr(|L_1(\mathcal G)-\theta(\lambda)n| \ge A\sqrt{n})\le Ce^{-cA^2}.
$$   
\end{lemma}
Let $\mathcal{L}_1(\mathcal G)$ denote the largest component of $\mathcal G$ (note that $|\mathcal{L}_1(\mathcal G)| = L_1(\mathcal G)$).
Lemma~\ref{lemma:rg:gc} and Hoeffding's inequality, which also holds when sampling without replacement (see~\cite{Hoeffding}) imply the following.
\begin{lemma}
    \label{lemma:rg:partition}
    Let $U \subseteq V$. For $\lambda > 1$, there exist constants $c = c(\lambda) > 0$ and $C > 0$ such that for any $A>0$ 
$$
\Pr(||{\mathcal L}_1(\mathcal G)\cap U|-\theta(\lambda)|U|| \ge A\sqrt{n})\le Ce^{-cA^2}.
$$
\end{lemma}

Finally, Theorem 1.1 in~\cite{janson2008susceptibility} and the reflection principle yields the following lemma.

\begin{lemma} \label{lemma:rg:sucep} \
\begin{enumerate} 
    \item If $\lambda < 1$, there exists $\gamma = \gamma(\lambda)$ such that $\E[\sum_{j \ge 1} L_j(\mathcal G)^2] = \gamma n + O(1)$ and $\Var[\sum_{j \ge 1} L_j(\mathcal G)^2] = O(n)$.
    \item If $\lambda > 1$, there exists $\gamma = \gamma(\lambda)$ such that $\E[\sum_{j \ge 2} L_j(\mathcal G)^2] = \gamma n + O(1)$ and $\Var[\sum_{j \ge 2} L_j(\mathcal G)^2] = O(n)$.
\end{enumerate}    
\end{lemma}

\subsection{Drift function}
\label{subsec:prelim}

Let us define the function $F: [1/q,1] \rightarrow [0,1]$ as
$$
F(x) = \frac{1}{q} + \Big(1-\frac{1}{q}\Big) \theta(\beta x) x,
$$
with $\theta(\beta x)$ given by~\eqref{eq:rg:root}.
This function plays a central role in the analysis of the mixing time of the SW dynamics, as
it captures the expected fraction of vertices in the largest color class after one step of the SW dynamics from a configuration with an $x$ fraction of the vertices in its largest color class. As such, 
we refer to $F$ as the ``drift function''.
Our proofs will use the following facts about $F$; the first two were previously established in~\cite{galanis2019swendsen} and we provide a proof of the next two in Appendix~\ref{appendix:f}.

\begin{fact}[Lemma 4~\cite{galanis2019swendsen}]
    \label{fact:df:fp}
    For $\beta > q$, $F$ has a unique fixed point $a(\beta,q) \in (1/q,1]$ and $|F'(a(\beta,q))| < 1$. 
\end{fact}

\begin{fact}[Lemma 8~\cite{galanis2019swendsen}] 
    \label{fact:f:inc}
    For $\beta > 0$, $F'(x) > 0$ and $F''(x) < 0$ for all $x \in (1/\beta,1]$. That is, $F$ is strictly increasing and concave in the interval $(1/\beta,1]$.
\end{fact}

\begin{fact} 
    \label{fact:f:d2}
    There exists $\Delta(q,\beta) > 0$ such that 
    $|F''(x)| \le \Delta(q,\beta)$ for all $x \in (0,1]$. 
\end{fact}

\begin{fact}
\label{fact:f:d:exp}
    $F'(a) = \frac{q-1}{q} \cdot \frac{\theta(a\beta)^2}{\theta(a\beta)+(1- \theta(a\beta))\log(1- \theta(a\beta)) }\ge \frac{q-1}{q}.$ 
\end{fact}

\section{Proof of main result: cutoff when $\beta > q$}

We assume through the remainder of the paper that $\beta>\bS=q$. 
Let $a = a(\beta,q) > 1/q$ be the unique fixed point of the drift function $F$ in $(1/q,1]$; see Fact~\ref{fact:df:fp}.
We use in our analysis two different low-dimensional projections of the SW dynamics. The first is the $q$-dimensional \emph{proportions} vector.  
For a configuration $X \in \Omega$, let $\alpha(X)$ be a $q$-dimensional vector, where each coordinate of $\alpha(X)$ corresponds to the fraction of vertices assigned a given spin in $X$. We will assume that $\alpha_1(X)$ contains the fraction of vertices assigned the dominant spin in~$X$, breaking ties arbitrarily; that is, $\alpha_1(X) \ge \alpha_i(X)$ for all $i \in [q]$.

When $\beta > \bS$, for a configuration $X \sim \mu$, $\alpha(X)$ will concentrate around
the $q$-dimensional \emph{majority} vector
$$
m := \Big(a,\frac{1-a}{q-1},\dots,\frac{1-a}{q-1} \Big)
$$
or one of its permutations. 

We will design a multi-phase coupling, and in the first two phases, we bound the time it takes a copy of the SW dynamics to reach the neighborhood of $m$ from an arbitrary initial configuration. We use the first phase to argue that the SW dynamics reaches constant distance from $m$ and the second phase to further contract the distance to $O(1/\sqrt{n})$; the latter will dominate the order of the mixing time.

Let $X_0 \in \Omega$ be an arbitrary configuration and $Y_0 \sim \mu$.
We will bound the probability that $\{X_t\}$ 
 and $\{Y_t\}$ have not coupled after $c(\beta,q) \log n + O(1)$ steps, where 
 \begin{equation}
 \label{eq:const}
 c(\beta,q) =  \Big(2 \log \Big(\frac{1}{F'(a)}\Big)\Big)^{-1} = \frac{1}{2 \log\Big(\frac{q}{q-1} \cdot \frac{\theta(a\beta)+(1- \theta(a\beta))\log(1- \theta(a\beta)) }{\theta(a\beta)^2}\Big)}; 
\end{equation}
 see Fact~\ref{fact:f:d:exp}. From~\eqref{eq:rg:root}, we know that $\theta(a\beta)$ is the unique root in $(0,1]$ of the equation
 $$
 e^{\frac{-\beta y}{q-(q-1)y}} = 1 - y.
 $$

In the first two phases $\{X_t\}_{t\ge 0}$ and $\{Y_t\}_{t\ge 0}$ evolve independently. 
We claim that after $T_1=O(1)$ steps
$\|\alpha(X_{T_1})-m\|_\infty\le \varepsilon$ and $\|\alpha(Y_{T_1})-m\|_\infty\le \varepsilon$ for any desired constant $\varepsilon \in (0,1)$.

\begin{lemma}
    \label{lemma:constant_distance}
        For any constants $\delta\in(0,1)$ and $\varepsilon > 0$ sufficiently small, for all sufficiently large $n$ and any starting state $X_0 \in \Omega$, after $T=O(1)$ steps, with probability at least $\delta$, the SW dynamics reaches a state $X_T$ such that $\|\alpha(X_T)-m\|_\infty\le \varepsilon$.
\end{lemma}

A version of this lemma was established in~\cite{galanis2019swendsen}, but we require a slight refinement here so that the probability of success is any constant arbitrarily close to $1$.

In the second phase, 
we claim that after an additional $T_2 = c(\beta,q) \log n + O(1)$ steps, in which both copies of the chain continue to evolve independently, for any constant $\delta \in (0,1)$ there exists a constant $C = C(\delta) > 0$ such that with with probability at least~$\delta$:
\begin{equation}
\label{eq:m:dist}
\|\alpha(X_{T_1+T_2}) - m\|_\infty \le \frac{C}{\sqrt{n}},~\text{and}~\|\alpha(Y_{T_1+T_2}) - m\|_\infty \le \frac{C}{\sqrt{n}}.
\end{equation}
This is established in the following lemma.
\begin{lemma}
     \label{lemma:main:contract-over}   
    Suppose $X_0$ is such that 
    $\|\alpha(X_0)-m\|_{\infty} \le \varepsilon$ for a sufficiently small constant $\varepsilon > 0$
    and let $T = c(\beta,q){\log n} + \gamma$ for constant $\gamma > 0$.
    Then, for any constant $\delta \in (0,1)$, there exists a constant $C = C(\gamma,\delta)>0$ such that with probability at least $\delta$, we have
    $
        \|\alpha(X_T) - m \|_{\infty} \le {C}/{\sqrt{n}}.
    $
\end{lemma}
    A similar convergence result was proved in~\cite{galanis2019swendsen}, where it is required that $T = A \log n$ with $A > c(\beta,q)$. This allowed for a simpler analysis. To obtain our sharper bound, we use the fact that deviations from the expected proportions vector after one step are amortized over the remaining steps of the dynamics, so that early (larger) deviations end up contributing less to the overall separation from $m$.

In the analysis of the third phase of coupling,
we use an additional ``stationary-like'' property 
for both copies of the chain. 
Let $\mathcal V = \{V_1,\dots,V_q\}$ be a partition of $V$ where $V_i$ contains all vertices of $X_{T_1}$ assigned spin $i$.
For constant $\lambda \in (0,1)$, we say that $\mathcal V$ is a $\lambda$-partition if $\min_i |V_i| \ge \lambda |V|$. From Lemma~\ref{lemma:constant_distance} we know that
$\mathcal V$ is a $\lambda$-partition with constant probability  arbitrarily close to $1$.

Let
$\alpha_{1,i}(X_t)$
denote the fraction of vertices of $V_i$ assigned the dominant spin in $X_t$; i.e., there are $\alpha_{1,i}(X_t)|V_i|$ such vertices.
The following lemma ensures that during the second phase, in parallel to making progress towards $m$, at some time $T_1 + T_2$, the fraction of vertices assigned the spin from the dominant spin class in each $V_i$ is roughly the same fraction as that from the full configuration $X_T$ in $V$.
This additional property will allow us to couple
the two process in $O(1)$ additional steps in the third phase. 

 \begin{lemma}
    \label{lemma:partition:convergence}
        Let $X_0$ be an arbitrary initial configuration and let $\mathcal V$ be a $\lambda$-partition of $V$ with constant $\lambda \in (0,1)$.  
        For any $\delta \in (0,1)$, there exist a constant $C > 0$ such that with probability at least $\delta$ when $T \ge \frac{\log n}{2 \log \big( \frac{q}{q-1}\big)} + \frac{\log (2/C)}{\log \big( \frac{q}{q-1}\big)}$, we have        
        $
        \max_{k\in[q]} |\alpha_1(X_T) - \alpha_{1,k}(X_T)| \le {C}/{\sqrt{n}}.
        $
\end{lemma}
We note that $(2 \log \big( \frac{q}{q-1}\big))^{-1} < c(\beta,q)$ by Fact~\eqref{fact:f:d:exp}, and so at time $T_1+T_2$ we reached two configurations that are close to $m$, i.e., satisfy~\eqref{eq:m:dist}, but that also have $\approx \!a|V_i|$ vertices assigned the dominant spin in each partition set $V_i$.

For the third phase, we consider another lower-dimensional projection of the SW dynamics; in particular, we zoom into a projection of dimension $q \times q$. For $\{X_t\}_{t \ge 0}$, let $A_t := \{A_{ij,t}\}_{i,j\in[q]},$
with $A_{ij,t}$ denoting the number of vertices in $V_i$ assigned color $j$ in $X_t$.
Similarly, define $A_t'$ for $\{Y_t\}_{t \ge 0}$.

We have shown that after the first two phases of the coupling, for any
constant $\hat\delta \in (0,1)$ there is a suitable constant $\hat C > 0$ such that at time $T = T_1+T_2$ with probability at least $\hat \delta$ the following holds:
\begin{enumerate}
        \item[A1.] $\|\alpha(X_T) - m\|_\infty \le \hat C /\sqrt{n}$ and $\|\alpha(Y_T) - m\|_\infty \le \hat C /\sqrt{n}$;
        \item[A2.]  $|\alpha_{1,i}(X_T) - \alpha_{1,i}(Y_T)| \le  \hat C / \sqrt{V_i}$ for all $i \in [q]$.
\end{enumerate}

Under these assumptions, we design a coupling such that $A_{T+1} = A_{T+1}'$ with probability $\Omega(1)$.

\begin{lemma}
    \label{lem:basket_coupling}
    Let $\lambda \in (0,1)$ be a constant independent of $n$.
    Let $\mathcal V = \{V_1,..., V_q\}$ be a $\lambda$-partition and let $A_t$ and $A_t'$ be two copies of the projection chain such that {\normalfont{A1}} and {\normalfont{A2}} above hold. 
     Then, there exists a coupling of $A_{t+1}$ and $A_{t+1}'$ such that with probability $\Omega(1)$, it holds that $A_{t+1}=A_{t+1}'$.
\end{lemma}

Similar couplings were also 
designed in~\cite{LNNP,galanis2019swendsen} but for the high-temperature $\beta < \bs$ setting where all the 
spin classes have roughly the same number of vertices and the percolation step is sub-critical in all classes. Here, we have to contend with the presence of a dominant spin class and a super-critical percolation step; this is the reason for the additional step in our analysis (Lemma~\ref{lem:basket_coupling}) that guarantees that the extra assumption A2 holds after the first two phases.

The final step of our coupling is a simple boosting phase so that the overall probability of success is at least $\ge 3/4$.

\begin{lemma}
\label{lemma:boosting}
    Let $X_0$ and $Y_0$ be configurations such that
    {\normalfont{A1}} and {\normalfont{A2}} hold. Then, there exists a coupling and a constant integer $\ell \ge 0$ such that 
    $
        \Pr\big(\cup_{i=1}^\ell \{A_{t+i}=A'_{t+i}\}\big)\ge \frac 34.
    $
\end{lemma}

In summary, we have shown that there exists a $T = c(\beta,q) \log n  + O(1)$ such that
$
\|A_{T} - A_{T}'\|_{\textsc{tv}} \le 1/4.
$
It was already noted in~\cite{LNNP,galanis2019swendsen} that the symmetry of the mean-field setting implies that
$\|X_{T}-Y_{T}\| = \|A_{T} - A_{T}'\|_{\textsc{tv}}$ and thus 
$$
\tau_{\rm mix}^{\mathrm{SW}} \le T = c(\beta,q)\log n+ O(1),
$$
which completes the proof of the upper bound in Theorem~\ref{thm:main:intro}.
The lower bound in this theorem follows straightforwardly from
some of the same facts required to establish Lemma~\ref{lemma:main:contract-over}, and it is provided later in Section~\ref{subsec:lb}.

We comment briefly in the organization of the rest of the paper. The proofs of Lemmas~\ref{lemma:constant_distance} and~\ref{lemma:main:contract-over}, as well as that of the mixing time lower bound, are provided in Section~\ref{sec:conv}. 
Section~\ref{sec:sf} contains the proof of Lemma~\ref{lemma:partition:convergence}, and Section~\ref{sec:pc} 
those of Lemmas~\ref{lem:basket_coupling} and \ref{lemma:boosting}.

\section{Convergence to the neighborhood of the majority vector}
\label{sec:conv}

We provide next the the proofs of Lemmas~\ref{lemma:constant_distance} and ~\ref{lemma:main:contract-over}, which ensure that SW dynamics reaches an $O(1/\sqrt{n})$ ball around $m$ after $T = c(\beta,q) \log n + O(1)$ steps. 
For $b \in (0,1)$, let 
$$
\kappa(b) := \Big(b,\frac{1-b}{q-1},\dots,\frac{1-b}{q-1}\Big).
$$
Let $F^k(x) = F^{k-1}(F(x))$ with $F^1 = F$.
We show first that $\alpha(X_t)$ is close to $\kappa(F^t(\alpha_1(X_0)))$ for $t \ge 0$ and then 
that, for an appropriate choice of $t$, $\kappa(F^t(\alpha_1(X_0)))$ is close to $m$.

\begin{lemma}
     \label{lemma:main:contract}   
    Suppose $X_0$ is such that 
    $\|\alpha(X_0)-m\|_{\infty}\le\varepsilon$ for a sufficiently small constant $\varepsilon > 0$. Then, 
    there exists a constant $\hat c(\beta,q) > c(\beta,q)$, such that 
    for any $t \le \hat c(\beta,q) \log n$ and any constant $\delta \in (0,1)$, there exists a constant $C(\delta)>0$ such that with probability at least $\delta$ 
    \begin{equation*}       
        \|\alpha(X_{t})-\kappa(F^{t}(\alpha_1(X_0)))\|_\infty \le \frac{C}{\sqrt{n}}.
    \end{equation*}
\end{lemma}

\begin{lemma}
    \label{lemma:f:fp-convergence}
    Suppose $x_0 \in (a-\varepsilon,a+\varepsilon)$ where $\varepsilon > 0$ is small enough and let $\gamma \in \R$. Then, if $T = c(\beta,q){\log n} + \gamma$, there exist constants $C_1 = C_1(\beta,q,\varepsilon)>0$ and $C_2 = C_2(\beta,q,\varepsilon)>0$ such that
    $$\frac{|x_0-a| F'(a)^\gamma C_1}{\sqrt{n}} \le |F^T(x_0) - a| \le \frac{|x_0-a|F'(a)^\gamma C_2 }{\sqrt{n}}.$$
\end{lemma}

A key step in the proof of Lemma~\ref{lemma:main:contract}, which is also useful in the proof of Lemma~\ref{lemma:constant_distance} is the following concentration guarantee on the fluctuations.
\begin{lemma}
    \label{lemma:step:concentration}
    Suppose $X_t$ is such that 
    $\|\alpha(X_t)-m\|_{\infty} \le \varepsilon$ for a sufficiently small $\varepsilon > 0$. Then, there exists $\gamma_0(\beta)$ such that for any $r \ge \gamma_0(\beta)$
    \begin{equation*}
        \Pr\Big(\|\alpha(X_{t+1})-\kappa(F(\alpha_1(X_{t})))\|_{\infty} \ge \frac{r}{\sqrt{n}}\Big)\le e^{-\Omega(r)}+O(r^{-2}).
    \end{equation*}
\end{lemma}

Before providing the proofs of these lemmas, 
we use them to prove Lemmas~\ref{lemma:constant_distance} and ~\ref{lemma:main:contract-over}.

\begin{proof}[Proof Lemma~\ref{lemma:constant_distance}]
    From Lemma 30 in~\cite{galanis2019swendsen}, we have that for any constant $\varepsilon>0$ and any starting state $X_0 \in \Omega$, there exists $T=O(1)$ such that after $T$ steps, we have that
    $\|\alpha(X_{T})-m\|_\infty\le \varepsilon$
    with probability at least $\gamma = \Omega(1)$.
    Running the dynamics for $kT$ steps, ensures the probability that we do not reach a configuration 
    $X$ such that $\|\alpha(X)-m\|_\infty\le \varepsilon$
    after $kT$ steps is at most $(1-\gamma)^k$.

    To complete the proof, we 
    we show that once we reach a configuration $\|\alpha(X_{t})-m\|_\infty\le \varepsilon$, then the same is true $X_{t+1}$ with high probability provided $\varepsilon$ is small enough. The result then follows from a union bound over the steps.
    If $\|\alpha(X_{t})-m\|_\infty\le \varepsilon$, then $\alpha_1(X_t) \in [a-\varepsilon,a+\varepsilon]$. 
    By Fact~\ref{fact:df:fp}, for $\varepsilon$ small enough there exists a constant $\rho \in (0,1)$ such that
    $|F(a+\varepsilon) - a| \le \rho \varepsilon$ 
    and $|F(a-\varepsilon) - a| \le \rho \varepsilon$.
    Since $F$ is increasing and concave in $(1/q,1)$ (see Fact~\ref{fact:f:inc}), 
    this implies that $(a+\varepsilon) - F(a+\varepsilon) \ge  (1-\rho)\varepsilon$ and
    $F(a-\varepsilon) - (a-\varepsilon) \ge (1-\rho)\varepsilon$. 
    Thus, in order for $\|\alpha(X_{t+1}) - m\|_\infty > \varepsilon$, we would need 
    $|\alpha_1(X_{t+1}) - F(\alpha_1(X_{t}))| \ge (1-\rho)\varepsilon$,
    which implies that 
    $\|\alpha(X_{t+1}) - \kappa(F(\alpha_1(X_{t})))\|_\infty \ge (1-\rho)\varepsilon$. 
    Then, from Lemma~\ref{lemma:step:concentration} we have that
    $$
       \Pr\Big(\|\alpha(X_{t+1})-\kappa(F(\alpha_1(X_{t})))\|_{\infty} \ge (1-\rho)\varepsilon \Big)\le O\big(n^{-1}\big).
     $$
\end{proof}

\begin{proof}[Proof of Lemma~\ref{lemma:main:contract-over}]
    Follows from~Lemmas~\ref{lemma:main:contract} and~\ref{lemma:f:fp-convergence} and the triangle inequality.
\end{proof}

\subsection{Proofs of Lemmas~\ref{lemma:main:contract}, ~\ref{lemma:f:fp-convergence}, and ~\ref{lemma:step:concentration}}

We provide first the proof of and Lemma~\ref{lemma:step:concentration}.

\begin{proof}[Proof of Lemma~\ref{lemma:step:concentration}]
Since $\beta > q$ and $a > 1/q$ by Fact~\ref{fact:df:fp}, $a \beta > 1$ and $\frac{1-a}{q-1} \beta < 1$. By assumption 
$\|\alpha(X_t)-m\|_{\infty}\le\varepsilon$, so when $\varepsilon$ is small enough, exactly one color class in the percolation step of the SW dynamics is supercritical, and the remaining $q-1$ color classes are subcritical.

Let $C_1, C_2, \dots$ denote the connected components after the percolation step from $X_t$, sorted in decreasing order by size (breaking ties arbitrarily) and let $R = \sum_{i\ge 2}|C_i|^2$.
By~Lemma~\ref{lemma:rg:sucep} there exist constants $\gamma_0,\gamma_1 > 0$, that depend only on $\beta$ and $q$, such that $\E[R] = \gamma_0 n + O(1)$ and
$
\Var[R] \le \gamma_1n.
$
Then, by Chebyshev's inequality, for any $\gamma_2 \ge 1$ and $n$ large we have
\begin{equation}
\Pr\left(\left|R-\gamma_0n\right| \ge \gamma_2\sqrt{\gamma_1 n}\right) 
\le
\Pr\left(\left|R-\E[R]\right| \ge \frac{\gamma_2}2\sqrt{\gamma_1 n}\right) 
\le \frac{4}{\gamma_2^2}.
\label{eq:gam1}
\end{equation}

For $j \ge 2$ and $i \in [q]$,
consider the random variables $Q_j^{(i)}$ where
\[
    Q_j^{(i)}= \begin{cases}
    |C_j| & \text{if }C_j\text{ colored }i,~\text{and}\\
    0 & \text{otherwise}.\\
    \end{cases}
\]
Let $Z_i = \sum_{j \ge 2} Q_j^{(i)}$. By Hoeffding's inequality, for any $\gamma_3 \ge 0$
\begin{equation*}
\Pr\Big(|Z_i-E[Z_i]| \ge \gamma_3\sqrt{n} \mid R \le \gamma_0n + \gamma_2\sqrt{\gamma_1n} \Big)\le 2\exp\Big(\frac{-2 \gamma_3^2}{\gamma_0+\gamma_2\sqrt{\gamma_1/n}}\Big),
\end{equation*}
and combined with~\eqref{eq:gam1} we deduce that
$$
\Pr\big[|Z_i-E[Z_i]| \ge \gamma_3\sqrt{n} \big]\le  2\exp\Big(\frac{-\gamma_3^2}{\gamma_0+\gamma_2\sqrt{\gamma_1/n}}\Big)  +\frac{4}{\gamma_2^2}.
$$

Turning our attention to $|C_1|$, by Lemma~\ref{lemma:rg:gc}, there are exists constants $c_1,c_2 > 0$ such that for any $\gamma_4 > 0$
\begin{equation}
    \Pr(||C_1|-\theta(\beta \alpha_1(X_t))\alpha_1(X_t)n|\ge\gamma_4\sqrt{n})\le c_1 \exp(-c_2 \gamma_4^2).
    \label{eq:gam4}
\end{equation}
If 
$||C_1|-\theta(\beta \alpha_1(X_t))\alpha_1(X_t)n| <\gamma_4\sqrt{n}$, and $|Z_j-E[Z_j]|<\gamma_3\sqrt{n}$ for all $i \in [q]$, we have
 \begin{align*}
 \alpha_1(X_{t+1})n &\le \theta(\beta \alpha_1(X_t))\alpha_1(X_t)n + \frac{n -  \theta(\beta \alpha_1(X_t))\alpha_1(X_t)n}{q}+  (\gamma_3+\gamma_4)\sqrt{n} = F(\alpha_1(X_t))n+ (\gamma_3+\gamma_4)\sqrt{n},  
 \end{align*}
 and similarly
 $$
 \alpha_1(X_{t+1})n \ge F(\alpha_1(X_t))n- (\gamma_3+\gamma_4)\sqrt{n}.
 $$
 In addition, for $j \neq 1$
\begin{align*}
 \alpha_j(X_{t+1})n  &\le \frac{n-\theta(\beta \alpha_1(X_t))\alpha_1(X_t)n}{q} +(\gamma_3+\gamma_4)\sqrt{n} 
 \le \frac{(1 - F(\alpha_1(X_t)))n}{q-1}+(\gamma_3+\gamma_4)\sqrt{n}, 
\end{align*}
and 
\begin{align*}
 \alpha_j(X_{t+1})n  \ge \frac{(1 - F(\alpha_1(X_t)))n}{q-1}-(\gamma_3+\gamma_4)\sqrt{n}.
\end{align*}

Combining these facts via a union bound, and setting $\gamma_2 = \gamma_3 = \gamma_4 = r/2$, it follows that for $r \ge \gamma_0$ and $n$ large enough that 
\begin{align*}
 \Pr\Big(\|\alpha(X_{t+1})-\kappa(F(\alpha_1(X_{t})))\|_{\infty} \ge \frac{r}{\sqrt{n}} \Big) 
 &\le c_1 \exp(-c_2 \gamma_4^2) +  2q\exp\Big(\frac{-\gamma_3^2}{\gamma_0+\gamma_2\sqrt{\gamma_1/n}}\Big)  +\frac{4q}{\gamma_2^2} \\
 &\le \exp(-\Omega(r)) + O(r^{-2})
 \end{align*}
 as claimed.
\end{proof}

We are now ready to provide the proof of Lemma~\ref{lemma:main:contract}.

\begin{proof}[Proof of Lemma~\ref{lemma:main:contract}]
    Let
    \begin{equation}
    \label{eq:xidef}
        \xi_t=n \cdot \|\alpha(X_t)-\kappa(F(\alpha_1(X_{t-1})))\|_\infty.
    \end{equation}
    By Fact~\ref{fact:df:fp}, $|F'(a)| < 1 $, so for small enough $\varepsilon$, there exists $\rho \in (0,1)$ such that
    $1-\rho = \max_{|x-a|<2\varepsilon}|F'(x)|$.
    Let $K > 0$ be a large constant we will choose later and let us define the sequence
    \begin{equation}
        r_i=K\big(1-{\rho}/{2}\big)^{t-i}\big(1+\rho\big)^{t-i}.
    \end{equation}
    Our first observation is that when $\xi_i <  r_i\sqrt{n}$ for all $i \le t$, we have  
    \begin{equation}
    \label{eq:taylor_prep}
    |F^{t-j-1}(\alpha_1(X_j))-a| < 2\varepsilon
    \end{equation}
    for all $j \le t-1$.
    To see this, note
    that since $F$ is increasing, concave and have a unique fixed point in $[1/q,1]$
    (see Facts~\ref{fact:df:fp} and~\ref{fact:f:inc})
    we have that for any $y \in [1/q,1]$
    \begin{equation}
        \label{eq:ineq:c}
        |F(y)-a| \le |y - a|.
    \end{equation}
    Hence,
    \begin{align*}
        |F^{t-j-1}(\alpha_1(X_j))-a| \le |\alpha_1(X_j)-a|,  
    \end{align*}
    and from \eqref{eq:xidef}  and the triangle inequality it follows that
       \begin{align*}
        |\alpha_1(X_j)-a| \le |F(\alpha_1(X_{j-1}))-a| + \frac{\xi_j}{n}.
    \end{align*}
    Using~\eqref{eq:ineq:c} again and combining these inequalities we conclude that:
    \begin{align*}
         |F^{t-j-1}(\alpha_1(X_j))-a| \le |\alpha_1(X_{j-1})-a| + \frac{\xi_j}{n}.  
    \end{align*}    
    Iterating this process:
    $$
    |F^{t-j-1}(\alpha_1(X_j))-a| \le |\alpha_1(X_0) - a| + \frac{1}{n}\sum_{i=1}^j \xi_i \le \varepsilon + \frac{1}{\sqrt{n}}\sum_{i=1}^j r_i.
    $$
    Observe that for a suitable constant $K' > 0$,
    $$
    \frac{1}{\sqrt{n}}\sum_{i=1}^j r_i =  \frac{K}{\sqrt{n}}\sum_{i=1}^t \big(1-{\rho}/{2}\big)^{t-i}\big(1+\rho\big)^{t-i} \le \frac{K' (\big(1-{\rho}/{2}\big)\big(1+\rho\big))^{t}}{\sqrt{n}}.
    $$
    So, when $t$ is such that $(\big(1-{\rho}/{2}\big)\big(1+\rho\big))^{t} 
    \le n^{\frac{1}{2}-\gamma}$ for a constant $\gamma \in (1/4,1/2)$, we have
    $$
    \frac{1}{\sqrt{n}}\sum_{i=1}^j r_i = o(1),
    $$ and thus
    $
    |F^{t-j-1}(\alpha_1(X_j))-a| \le \varepsilon + o(1) \le 2\varepsilon
    $
    for $n$ large which establishes~\eqref{eq:taylor_prep}.
    Observe that 
    $$(\big(1-{\rho}/{2}\big)\big(1+\rho\big))^{t} 
    \le n^{\frac{1}{2}-\gamma}$$
    when $t \le \frac{\frac{1}{2} - \gamma}{\log ((1-{\rho}/{2})(1+\rho))} \log n$, and it can be checked that 
    $$
    \frac{\frac{1}{2} - \gamma}{\log ((1-{\rho}/{2})(1+\rho))} > c(\beta,q) = \frac{1}{2 \log \frac{1}{F'(a)}}.
    $$
    
    Now, going back to~\eqref{eq:xidef},
    \begin{align*}
         F(\alpha_1(X_{t-1}))-\frac{\xi_{t}}{n} \le \alpha_1(X_{t})
           \le F(\alpha_1(X_{t-1}))+\frac{\xi_{t}}{n}.
    \end{align*}
    Similarly, 
   \begin{equation} \label{eq:c1}
    F(\alpha_1(X_{t-2}))-\frac{\xi_{t-1}}{n} \le \alpha_1(X_{t-1})
           \le F(\alpha_1(X_{t-2}))+\frac{\xi_{t-1}}{n},
    \end{equation}
    and since $F$ is increasing in $[1/q,1]$ (see Fact~\ref{fact:f:inc}), we have
    $$
   F\Big(F(\alpha_1(X_{t-2}))-\frac{\xi_{t-1}}{n}\Big) - \frac{\xi_{t}}{n} \le \alpha_1(X_{t})
           \le F\Big(F(\alpha_1(X_{t-2}))+\frac{\xi_{t-1}}{n}\Big) + \frac{\xi_{t}}{n}.       
    $$
    Since from~\eqref{eq:taylor_prep} we know that
    $F(\alpha_1(X_{t-2})) \in (a-2\varepsilon,a+2\varepsilon)$.     
    Hence, the Taylor expansion of the function $F$ about $F(\alpha_1(X_{t-2}))$ yields that 
    
    \begin{equation}
    \label{eq:c2}
     F\Big(F(\alpha_1(X_{t-2}))+\frac{\xi_{t-1}}{n}\Big) \le F(F(\alpha_1(X_{t-2}))) +  \frac{(1-\rho) \xi_{t-1}}{n},
    \end{equation}
    and
    \begin{equation}
     \label{eq:c3}
    F\Big(F(\alpha_1(X_{t-2}))-\frac{\xi_{t-1}}{n}\Big) \ge F(F(\alpha_1(X_{t-2}))) -  \frac{(1-\rho) \xi_{t-1}}{n}.
    \end{equation}
    Therefore,
    \begin{equation*}
         F(F(\alpha_1(X_{t-2}))) -  \frac{(1-\rho) \xi_{t-1}}{n} - \frac{\xi_{t}}{n} \le \alpha_1(X_t) \le F(F(\alpha_1(X_{t-2}))) +  \frac{(1-\rho) \xi_{t-1}}{n} + \frac{\xi_{t}}{n}.
    \end{equation*}
    Iterating this argument, making use of ~\eqref{eq:taylor_prep} in each step, we deduce that
    \begin{equation}
    \label{eq:c4}
        F^t(\alpha_1(X_0)) - \frac{1}{n}\sum_{i=1}^{t}(1-\rho)^{t-i}\xi_{i} \le \alpha_1(X_t) \le F^t(\alpha_1(X_0)) + \frac{1}{n}\sum_{i=1}^{t}(1-\rho)^{t-i}\xi_{i},
    \end{equation}
    so that 
    \begin{equation}
    \label{eq:c0}
        |\alpha_1(X_t)-F^t(\alpha_1(X_0))| \le \frac{1}{n} \sum_{i=1}^{t}(1-\rho)^{t-i}\xi_{i}.
   \end{equation}
    In a similar fashion, from~\eqref{eq:xidef} for $j \neq 1$
    \begin{align*}
         \frac{1-F(\alpha_1(X_{t-1}))}{q-1}-\frac{\xi_{t}}{n} \le \alpha_j(X_{t})
           \le \frac{1-F(\alpha_1(X_{t-1}))}{q-1}+\frac{\xi_{t}}{n},
    \end{align*}
    and from~\eqref{eq:c1} since $F$ is increasing
    $$
      \frac{1-F\big(F(\alpha_1(X_{t-2}))+\frac{\xi_{t-1}}{n}\big)}{q-1}-\frac{\xi_{t}}{n} \le \alpha_j(X_{t})
           \le \frac{1-F\big(F(\alpha_1(X_{t-2}))+\frac{\xi_{t-1}}{n}\big)}{q-1}+\frac{\xi_{t}}{n}.
    $$
    From~\eqref{eq:c2} and~\eqref{eq:c3}, we then deduce that
    $$
       \frac{1-F^2\big(\alpha_1(X_{t-2}))}{q-1} -\frac{(1-\rho)\xi_{t-1}}{(q-1)n}- \frac{\xi_{t}}{n} \le \alpha_j(X_{t})
           \le \frac{1-F^2\big(\alpha_1(X_{t-2}))}{q-1} +\frac{(1-\rho)\xi_{t-1}}{(q-1)n}+ \frac{\xi_{t}}{n},
    $$
    and iterating as in~\eqref{eq:c4}:
    \begin{equation}
    \label{eq:c5}
    \left|\alpha_j(X_t) - \frac{1-F^t\big(\alpha_1(X_{t-2}))}{q-1}\right| \le \frac{1}{(q-1)n} \sum_{i=1}^{t}(1-\rho)^{t-i}\xi_{i}.
    \end{equation}
    From~\eqref{eq:c0} and~\eqref{eq:c5}, it suffices to show that
    for a suitable constant $C >0$, with probability at least $\delta \in (0,1)$
    \[
        \sum_{i=1}^{t}(1-\rho)^{t-i}\xi_{i}\le C\sqrt{n}.
    \]
    When $\xi_i <  r_i\sqrt{n}$ for all $i \le t$, then
    $$
    \sum_{i=1}^{t}(1-\rho)^{t-i}\xi_{i}\le     K\sqrt{n} \sum_{i=1}^{t}(1-\rho)^{t-i} \big(1-{\rho}/{2}\big)^{t-i}\big(1+\rho\big)^{t-i} \le C\sqrt{n}
    $$
    for large enough $C$.
    By Lemma~\ref{lemma:step:concentration}, we can choose $K$ large enough such that
    \begin{equation*}
        \sum_{i=1}^t\Pr(\xi_i\ge r_i\sqrt{n})\le\sum_{i=1}^tO(r_i^{-2})\le \delta.
    \end{equation*}
    It then follows from a union bound that 
    \begin{equation*}
        \Pr\Big(\sum_{i=1}^{t}\xi_i(1-\rho)^{t-i}\ge C\sqrt{n}\Big)
        \le \sum_{i=1}^{t} \Pr(\xi_i > r_i\sqrt{n}) \le \delta,
    \end{equation*}
    as claimed.
\end{proof}

We complete this proof by proving the proof of Lemma~\ref{lemma:f:fp-convergence}.

\begin{proof}[Proof of Lemma~\ref{lemma:f:fp-convergence}]
   For ease of notation,  set $|F'(a)| = \eta < 1$. By Fact~\ref{fact:f:d2}, there exists $\Delta \ge |F''(x)|$ for all $x \in (a-\varepsilon,a+\varepsilon)$.
   Let $d_t = |a-F^t(x_0)|$. The Taylor expansion of $F$ about $a$ implies that
   \begin{equation}
    \label{eq:f:taylor}
        \eta d_{t-1} - \frac{\Delta}{2}d_{t-1}^2 \le d_{t} \le \eta d_{t-1} + \frac{\Delta}{2}d_{t-1}^2.
   \end{equation}
   Iterating,
   $$
   \eta^T d_{0} - \frac{\Delta}{2\eta} \sum_{i=1}^T \eta^id_{T-i}^2 \le d_T \le \eta^T d_{0} + \frac{\Delta}{2\eta} \sum_{i=1}^T \eta^id_{T-i}^2.
   $$
   Since $d_t \le \varepsilon$ for $t \ge 0$, for $\varepsilon$ small enough, we can deduce from~\eqref{eq:f:taylor} that 
   $$
   \big(\eta-\frac{\Delta\varepsilon}{2}\big)d_{t-1} \le d_t \le \big(\eta+\frac{\Delta\varepsilon}{2}\big)d_{t-1},$$ and thus
   \begin{align*}
   \frac{d_0 \eta^\gamma}{\sqrt{n}} - \frac{\Delta d_0^2(\eta-\frac{\Delta\varepsilon}{2})^{2T}}{2\eta} \sum_{i=1}^T \Big(\frac{\eta}{(\eta-\frac{\Delta\varepsilon}{2})^2}\Big)^i \le d_T 
    \le \frac{d_0 \eta^\gamma}{\sqrt{n}} + \frac{\Delta d_0^2(\eta+\frac{\Delta\varepsilon}{2})^{2T}}{2\eta} \sum_{i=1}^T \Big(\frac{\eta}{(\eta+\frac{\Delta\varepsilon}{2})^2}\Big)^i.
   \end{align*}   
   For $\varepsilon$ small enough, $\eta > (\eta + \frac{\Delta\varepsilon}{2})^2 >  (\eta - \frac{\Delta\varepsilon}{2})^2$, and thus there exist constants $\gamma_1 = \gamma_1(\eta,\Delta,\varepsilon) > 0$ and $\gamma_2 = \gamma_2(\eta,\Delta,\varepsilon) > 0$ such that
   $$
   d_T \le \frac{d_0 \eta^\gamma}{\sqrt{n}} + \frac{\Delta d_0^2\gamma_1}{2\eta}\frac{\eta^\gamma}{\sqrt{n}} = \frac{d_0 \eta^\gamma (1 + \frac{\Delta\varepsilon\gamma_1}{2\eta})}{\sqrt{n}},
   $$
   and
   $$
   d_T \ge \frac{d_0 \eta^\gamma}{\sqrt{n}} - \frac{\Delta d_0^2\gamma_2}{2\eta}\frac{\eta^\gamma}{\sqrt{n}} = \frac{d_0 \eta^\gamma (1 - \frac{\Delta\varepsilon\gamma_2}{2\eta})}{\sqrt{n}},
   $$
   as claimed.
\end{proof}

\subsection{Mixing time lower bound}
\label{subsec:lb}

We will use the following 
result for proving the the lower bound on the mixing time as stated in Theorem~\ref{thm:main:intro}.

\begin{lemma}
    \label{lemma:potts:conc}
    If $X \sim \mu$, then $\|\alpha(X)-m\|_{\infty} \le A/{\sqrt{n}}$ with probability at least $3/4$ for a suitable constant $A > 0$.
\end{lemma}
Similar results are available in the literature, but we were not able to find this particular variant, so, for completeness, we provide a proof of Lemma~\ref{lemma:potts:conc} in Appendix~\ref{app:potts:conc}.

\begin{proof}[Proof of Theorem~\ref{thm:main:intro} (Lower bound)]
    By Lemma~\ref{lemma:potts:conc},
    if $X \sim \mu$, then $\|\alpha(X)-m\|_{\infty} \le A/{\sqrt{n}}$ with probability at least $3/4$ for a suitable constant~$A$.
    Now, take $X_0$ to be a configuration where all vertices are assigned the same spin.
    By~Lemmas~\ref{lemma:main:contract} and~\ref{lemma:f:fp-convergence}, 
    there exists $\gamma$
    such that after $T = c(\beta,q)\log n + \gamma$ steps,
    with probability $\ge 3/4$, we have $\|\alpha(X_T)-m\|_\infty > \frac{A}{\sqrt{n}}$.    
    This implies that $\|X_T - \mu\|_{\textsc{tv}} > 1/4$ and thus that the SW dynamics has not mixed after $T$ steps.
\end{proof}

\section{Spin fraction coalescence in partition}
\label{sec:sf}
    We provide in this section the proof of Lemma~\ref{lemma:partition:convergence}. 
    For this, for $k \in [q]$, let $d_t^{(k)} = |\alpha_1(X_t)-\alpha_{1,k}(X_t)|$ and 
    \[
        \hat{\xi}_{t+1}^{(k)}=\Big|d_{t+1}^{(k)}-\Big(1-\frac{1}{q}\Big)\theta(\beta  \alpha_1(X_{t}))d_{t}^{(k)}\Big|.
    \]
    We will need the following concentration bound for $\hat{\xi}_{t}^{(k)}$.

     \begin{lemma}
        \label{lemma:step:concentration_basket}
        Let $\{V_1,\dots,V_q\}$ be a $\lambda$-partition of $V$ and suppose $X_t$ is such that 
        $\|\alpha(X_t)-m\|_{\infty} \le \varepsilon$ for a sufficiently small $\varepsilon > 0$. Then, there exists $\gamma_1(\beta,q) > 0$ such that for any $r \ge \gamma_1(\beta,q)$ and any $k \in [q]$
        \begin{equation*}
            \Pr\Big(\hat{\xi}_{t+1}^{(k)} \ge \frac{r}{\sqrt{n}}\Big)\le e^{-\Omega(r)}+O(r^{-2}).
        \end{equation*}
    \end{lemma}
        
    \begin{proof}
        Since $\beta > q$ and $a > 1/q$ by Fact~\ref{fact:df:fp}, $a \beta > 1$ and $\frac{\beta(1-a)}{q-1} < 1$. By assumption 
        $\|\alpha(X_t)-m\|_{\infty}\le\varepsilon$, so when $\varepsilon$ is small enough, exactly one color class in the percolation step of the SW dynamics is supercritical, and the remaining $q-1$ are subcritical.

        Without loss of generality, let us prove the statement for $V_1$ (i.e., for $k=1$), noting that the same argument applies to any other $V_i$.
        Let $C_1$, $C_2$,... denote the connected components after the percolation step from $X_t$, sorted in decreasing order by size,
        let $R = \sum_{i\ge 2}|C_i|^2$ and $R_1 = \sum_{i\ge 2}|C_i\cap V_1|^2$.
        By~Lemma~\ref{lemma:rg:sucep} there exist constants $\gamma_0,\gamma_1 > 0$ such that $\E[R] = \gamma_0 n + O(1)$ and        
        $\Var[R] \le \gamma_1n$.        
        Then, by Chebyshev's inequality, for any $\gamma_2 \ge 1$ and $n$ large we have
        \begin{equation}
        \Pr\big[R_1 \ge \gamma_0 n+ \gamma_2\sqrt{\gamma_1 n} \big] 
        \le
        \Pr\big[R \ge \gamma_0n +  {\gamma_2}\sqrt{\gamma_1 n}\big]  
        \le
        \Pr\Big(R \ge \E[R] +  \frac{\gamma_2}{2}\sqrt{\gamma_1 n}\Big)  
        \le \frac{4}{\gamma_2^2}.
        \label{eq:gam2}
        \end{equation}
        
        For $j \ge 2$ and $i \in [q]$,
        consider the random variables $Q_j^{(i)}$ where
        \[
            Q_j^{(i)}= \begin{cases}
            |C_j \cap V_1| & \text{if}~C_j\text{ is colored }i;~\text{and}\\
            0 & \text{otherwise}.\\
            \end{cases}
        \]
        Let $Z_i = \sum_{j \ge 2} Q_j^{(i)}$ be the number of vertices colored $i$ in $V_1$. By Hoeffding's inequality, for any $\gamma_3 \ge 0$:
        \begin{equation*}
        \Pr\Big(|Z_i-E[Z_i]| \ge \gamma_3\sqrt{n} \mid R_1 \le \gamma_0 n + \gamma_2\sqrt{\gamma_1n} \Big)\le 2\exp\Big(\frac{-2 \gamma_3^2}{\gamma_0+\gamma_2\sqrt{\gamma_1/n}}\Big),
        \end{equation*}
        and combined with~\eqref{eq:gam2} we obtain
        $$
        \Pr\big[|Z_i-E[Z_i]| \ge \gamma_3\sqrt{n} \big]\le  2\exp\Big(\frac{-\gamma_3^2}{\gamma_0+\gamma_2\sqrt{\gamma_1/n}}\Big)  +\frac{4}{\gamma_2^2}.
        $$
        We consider next $|C_1\cap V_1|$.
     By Lemma~\ref{lemma:rg:partition}, we have that for suitable constants $c_1,c_2 > 0$ for any $\gamma_4 > 0$ 
        \[
            \Pr\left(\big||C_1\cap V_1|-\theta(\beta \alpha_1(X_t))\alpha_{1,1}(X_t)|V_1|\big|\ge\gamma_4\sqrt{n}\right)\le c_1\exp{(-c_2\gamma_4^2)}.
        \]
        If
        $||C_1\cap V_1|-\theta(\beta \alpha_1(X_t))\alpha_{1,1}(X_t)|V_1|\big|<\gamma_4\sqrt{n}$, and $|Z_j-E[Z_j]|<\gamma_3\sqrt{n}$ for all $j \in [q]$, we have   \begin{align*}
        \alpha_{1,1}(X_{t+1})|V_1| &\le \theta(\beta \alpha_1(X_t))\alpha_{1,1}(X_t)|V_1| + \frac{|V_1| - \theta(\beta  \alpha_1(X_t))\alpha_{1,1}(X_t)|V_1|}{q}+  (\gamma_3+\gamma_4)\sqrt{n} \\&= G(\alpha_1(X_t),\alpha_{1,1}(X_{t}))|V_1|+ (\gamma_3+\gamma_4)\sqrt{n}, 
         \end{align*}
        where the function $G: [1/q,1] \times [1/q,1] \rightarrow  [0,1]$ is defined as
    \[
        G(x,y)=\frac 1q + \theta(\beta x)y\Big(1-\frac 1q\Big).
    \]         
         Similarly, we have that
         $$
        \alpha_{1,1}(X_{t+1})|V_1|  \ge G(\alpha_1(X_t),\alpha_{1,1}(X_{t+1}))|V_1|- (\gamma_3+\gamma_4)\sqrt{n},
         $$
         and for $j \neq 1$
         $$
         \frac{n- \theta(\beta \alpha_1(X_t))\alpha_{1,1}(X_t)|V_1|}{q} - (\gamma_3+\gamma_4)\sqrt{n} \le  \alpha_{j,1}(X_{t+1})|V_1|   \le \frac{ n - \theta(\beta \alpha_1(X_t))\alpha_{1,1}(X_t)|V_1|}{q} +(\gamma_3+\gamma_4)\sqrt{n}.
         $$
          Combining these facts
         and setting $\gamma_2 = \gamma_3 = \gamma_4 = r/4$, it follows that for $r \ge \gamma_0$ and $n$ large enough that 
        \begin{align}
        \Pr\Big(|\alpha_{1,1}(X_{t+1}) - G(\alpha_1(X_t),\alpha_{1,1}(X_{t}))| \ge  \frac{r}{2\sqrt{n}}\Big) &\le
          \Pr\Big(|\alpha_{1,1}(X_{t+1}) - G(\alpha_1(X_t),\alpha_{1,1}(X_{t}))| \ge  \frac{(\gamma_3+\gamma_4)}{\sqrt{n}}\Big) \notag\\
         &\le
         \Pr\Big(|\alpha_{1,1}(X_{t+1}) - G(\alpha_1(X_t),\alpha_{1,1}(X_{t}))| \ge  \frac{(\gamma_3+\gamma_4)\sqrt{n}}{|V_1|}\Big) \notag\\         
         &\le c_1 \exp(-c_2 \gamma_4^2) +  2q\exp\Big(\frac{-\gamma_3^2}{\gamma_0+\gamma_2\sqrt{\gamma_1/n}}\Big)  +\frac{4q}{\gamma_2^2}\notag\\
         &\le e^{-\Omega(r)}+O(r^{-2}). \label{eq:fcb1}
         \end{align}
        By Lemma~\ref{lemma:step:concentration}, for large enough $r$
        \begin{equation}
        \label{eq:fcb2}
        \Pr\Big(|\alpha_1(X_{t+1}) - F(\alpha_1(X_t))| \ge \frac{r}{2\sqrt{n}}\Big)\le e^{-\Omega(r)}+O(r^{-2}).
       \end{equation}
        Since
        \begin{align*}
        d_{t+1}^{(1)} &= |\alpha_1(X_{t+1}) - \alpha_{1,1}(X_{t+1})|,\\
         \hat{\xi}_{t+1}^{(1)}&=\Big|d_{t+1}^{(1)}-\Big(1-\frac{1}{q}\Big)\theta(\beta \alpha_1(X_{t}))d_{t}\Big|,~\text{and}\\
        F(x) - G(x,y) &= \Big(1 - \frac{1}{q}\Big)(x-y) \theta(\beta x),
        \end{align*}
        it follows from~\eqref{eq:fcb1} and~\eqref{eq:fcb2} that
        $$
        \Pr\Big(\hat{\xi}_{t+1}^{(1)} \ge \frac{r}{\sqrt{n}}\Big)\le e^{-\Omega(r)}+O(r^{-2}),
        $$       
        as claimed.
    \end{proof}

    We can now provide the proof of Lemma~\ref{lemma:partition:convergence}.
    
    \begin{proof}[Proof of Lemma~\ref{lemma:partition:convergence}]
    Recall that $d_t^{(k)} = |\alpha_1(X_t)-\alpha_{1,k}(X_t)|$ and 
    \[
        \hat{\xi}_{t}^{(k)}=\Big|d_{t}^{(k)}-\Big(1-\frac{1}{q}\Big)\theta(\beta  \alpha_1(X_{t-1}))d_{t-1}^{(k)}\Big|.
    \]
    Then,
    \begin{align*}
        d_{t}^{(k)}
            &\le    \Big(1-\frac{1}{q}\Big)\theta(\beta \alpha_1(X_{t-1}))d_{t-1}^{(k)} + \hat{\xi}_{t} ^{(k)}
            \le  \Big(1-\frac{1}{q}\Big)d_{t-1}^{(k)} + \hat{\xi}_{t}^{(k)}.  
    \end{align*}
    Iterating, we deduce that
    $$
    d_t^{(k)} \le \Big(1-\frac{1}{q}\Big)^t + \sum_{i=1}^{t}\Big(1-\frac{1}{q}\Big)^{t-i}\hat{\xi}_{i}^{(k)}.
    $$
    Let $\hat K > 0$, $\hat\rho = 1/q$, and define the sequence
    $
        \hat{r}_i=\hat K\big(1-{\hat\rho}/{2}\big)^{t-i}\big(1+\hat\rho\big)^{t-i}.
   $
    By Lemma~\ref{lemma:step:concentration_basket}, for any $\delta > 0$ there exists $\hat K$ large enough such that
    \begin{equation*}
         \sum_{i=1}^t\Pr\Big(\hat{\xi}_i^{(k)}\ge \frac{\hat{r}_i}{\sqrt{n}}\Big)\le\sum_{i=1}^tO(\hat{r}_i^{-2})\le \frac{\delta}{q}.
    \end{equation*}
    If $\hat{\xi}_i^{(k)} <  \hat{r}_i/\sqrt{n}$ for all $i \le t$, then
    $$
    \sum_{i=1}^{t}(1-\hat\rho)^{t-i}\hat \xi_{i}^{(k)} \le     \frac{\hat K}{\sqrt{n}} \sum_{i=1}^{t}(1-\hat\rho)^{t-i} \big(1-{\hat\rho}/{2}\big)^{t-i}\big(1+\hat\rho\big)^{t-i} \le \frac{C}{2\sqrt{n}},
    $$
    for a suitable constant $C > 0$. A union bound then implies that 
    \begin{equation*}
        \Pr\Big(\sum_{i=1}^{t}\hat \xi_i^{(k)}(1-\hat\rho)^{t-i}\ge \frac{C}{2\sqrt{n}}\Big)
        \le \sum_{i=1}^{t} \Pr\Big(\hat \xi_i^{(k)} > \frac{\hat  r_i}{\sqrt{n}}\Big) \le \frac{\delta}{q}.
    \end{equation*}
    Then, when $t \ge \frac{\log n}{2 \log \big( \frac{q}{q-1}\big)} + \frac{\log (2/C)}{\log \big( \frac{q}{q-1}\big)}$, we have that
    $d_t^{(k)} \le C/\sqrt{n}$
    with probability at least $1-\delta/q$, and the result follows by a union bound over $k \in [q]$.
    \end{proof}

\section{Projection chain}
\label{sec:pc}

In this section, we prove Lemma~\ref{lem:basket_coupling} using the following multinomial coupling from~\cite{galanis2019swendsen}. 

\begin{fact}
\label{fact:multinomial}

Fix integers $q\ge2$, $m \ge 0$
and a constant $L>0$. 
Let $X$ and $Y$ be two multinomial random variables with $m$ trials, where each trial leads to success of one of $q$ values with the probability of success for each value being $1/q$; i.e., $X,Y\sim\mathrm{Mult}(m;1/q,\ldots,1/q)$.
Let $w=(w_1,\dots,w_q) \in\mathbb{Z}^q$ such that $\sum_{i=1}^q w_i=0$ and $\|w\|_\infty\le L\sqrt{m}$. 
Then there exists a coupling of $X$ and $Y$ and a constant $c=c(q,L)>0$ independent of $m$ such that
\[
\Pr\big(X=Y+w\big)\ \ge\ c.
\]
\end{fact}

We can now provide the proof of Lemma~\ref{lem:basket_coupling}.

\begin{proof}[Proof of Lemma~\ref{lem:basket_coupling}]
    We couple $A_{t+1}$ and $A_{t+1}'$ as follows. The percolation step is done independently in both copies. In the coloring step, first, the same random color $s$ is assigned to the largest connected component in both chains, 
    and all other connected components with two or more vertices from both chains
    are assigned spins uniformly at random independently.
    In addition, in each $V_i$, if one copy has more isolated vertices than the other, the excess of isolated vertices in that copy are assigned spins independently.
    The coloring of the remaining isolated vertices will be coupled to fix the potential discrepancies created in each $V_i$ by this initial procedure as we described next.
    
    For $i,j\in[q]$, let $\hat{a}_{ij}$, $\hat{a}_{ij}'$ be the number of vertices assigned color $j$ in $V_i$ so far. We claim that there exists a constant $C>0$ such that with probability $\Omega(1)$ it holds for all $i,j\in[q]$ that
    \begin{equation}
        |\hat{a}_{ij}-\hat{a}_{ij}'|\le C\sqrt{|V_i|}.
        \label{eq:4}
    \end{equation}
    To prove this, let $\{C_k\}_{k\ge1}$ be the connected components in the first chain after the percolation step; assume $C_1$ is the largest component. Let $C(v)$ be the connected component of vertex $v$. Note that
    \[
        \sum_{k \ge 2}|C_k\cap V_i|^2=\sum_{k \ge 2}\sum_{v\in C_k\cap V_i}|C_k\cap V_i| = \sum_{k \ge 2}\sum_{v\in C_k\cap V_i}|C(v) \cap V_i| \le \sum_{v\in V_i\setminus C_1}|C(v)|.
    \]
    
   For $v \not\in C_1$,  $E[|C(v)|] \le A_1$ for a suitable constant $A_1 > 0$.
    Hence, taking expectations, it follows
    from Markov's inequality that 
    \[
        \Pr\left(\sum_{k \ge 2}|C_k\cap V_i|^2 \ge 100 q A_1 |V_i|\right) \le \frac{1}{100q}.
    \] 
    Let $\hat{n}_i$ be the number of vertices in $V_i$ excluding those in $C_1$ and the unassigned isolated vertices. 
    If $C_1$ is assigned spin $s$, then
    conditioning on $\sum_{k \ge 2}|C_k\cap V_i|^2 < 100 q A_1 |V_i|$,
    by Hoeffding's inequality, we deduce that for a large enough constant $A_2 > 0$, 
    with probability at least $\frac{99}{100q^2}$, for any $j \in[q]$ with $j \neq s$, it holds that
    \[
        \Big|\hat{a}_{ij}-\frac{\hat{n}_i}{q}\Big| \le A_2 \sqrt{|V_i|}.
    \]
    Similarly, we can get an analogous bound for $\hat{a}_{ij}'$, and by a union bound \eqref{eq:4} holds for all $i,j \in [q]$ such that $j \neq s$ with probability at least $96/100$.

    By Lemma~\ref{lemma:rg:partition} and assumption A2,  
    for a suitable constant $A_3 > 0$, we have for any $i \in [q]$ that
    $$
    \Pr\Big(||C_1 \cap V_i| - |C_1' \cap V_i| | \ge A_3 \sqrt{n} \Big) \le \frac{1}{100q}.
    $$
    Conditioning on this event and on the event that $\sum_{k \ge 2}|C_k\cap V_i|^2 < 100 q A_1 |V_i|$ both for all $i \in [q]$, 
    we obtain via Hoeffding's inequality that for a large enough constant $A_2' > 0$ that for any $i \in [q]$
        \begin{align*}
    \Big|\hat{a}_{is} - |C_1\cap V_i| - \frac{\hat{n}_s}{q}\Big| &\le  (A_2'+A_3) \sqrt{n},~\text{and}\\
    \Big|\hat{a}_{is}' - |C_1'\cap V_i| - \frac{\hat{n}_s}{q}\Big| &\le  (A_2'+A_3) \sqrt{n}
    \end{align*}
    with probability at least $99/(100q)$,
    which implies that \eqref{eq:4} holds for all $i \in [q]$ and $j = s$ with that probability $99/100$ by a union bound.

    Since $\mathcal V$ is a $\lambda$-partition,
    Lemma~\ref{lemma:rg:isolated} and a union bound over the partition sets imply that
    there are at least $\Omega(n)$ unassigned isolated vertices in each $V_i$ in both chains with probability $1-O(n^{-1})$.     
    Observe that the distribution of the coloring of the isolated vertices in each $V_i$ and each copy is a multinomial distribution.
    Then, using Fact~\ref{fact:multinomial}, we can couple the coloring of the isolated vertices in each $V_i$ to equalize the spin counts and obtain two configurations such that $A_{t+1}=A_{t+1}'$ with probability $\Omega(1)$.
\end{proof}

We wrap this section with the proof Lemma of~\ref{lemma:boosting}. The idea behind this proof is that we can show that once A1 and A2 are achieved, they are preserved for a constant number of steps with constant probability, providing more attempts for the coupling from Lemma~\ref{lem:basket_coupling} to succeed at least once. 

To formalize this, let $\mathcal P(r)$ be the set of pairs of configurations of the projection chain $A$ and $A'$ 
such that any full configurations 
$X$ and $X'$ that project 
onto $A$ or $A'$, respectively, satisfy:
$\|\alpha(X) - m\|_\infty \le r /\sqrt{n}$, $\|\alpha(Y) - m\|_\infty \le r /\sqrt{n}$,
$\max_{k \in [q]}|\alpha_{1,k}(X) - \alpha_{1,k}(Y)| \le r /\sqrt{n}$.
\begin{fact}
\label{fact:step:preserv}
    For any $\delta \in (0,1)$, there exists $r = r(\delta) > 0$ such that
    $\Pr(Z_{t+1} \in \mathcal P(r) \mid Z_t \in \mathcal{P}(r)) \ge \delta$.
\end{fact}
\begin{proof}
    Let $X_t$ and $Y_t$ be two configurations that project onto $Z_t=Z \in \mathcal{P}(r)$.
    Then,
    $\|\alpha(X_t) - m\|_\infty \le r /\sqrt{n}$.
    Let us assume that $\alpha_1(X_{t}) \ge a$ (the case when $\alpha_1(X_{t}) < a$ is analogous).
    By Fact~\ref{fact:f:inc}, $F(\alpha_1(X_{t})) \le F(a+r/\sqrt{n})$.
    Moreover, by Fact~\ref{fact:df:fp} and the mean value theorem, there exists a constant $\eta \in (0,1)$ such that
    $F(a+r/\sqrt{n}) - a = \eta r/\sqrt{n}$ which implies that $(a+r/\sqrt{n}) - F(a+r/\sqrt{n}) = (1-\eta)r/\sqrt{n}$. In order for $\|\alpha(X_{t+1}) - m\|_\infty \ge r /\sqrt{n}$, we would need 
    $\|\alpha(X_{t+1}) - \kappa(F(\alpha_1(X_{t})))\|_\infty \ge \frac{(1-\eta)r}{\sqrt{n}}$. 
    From Lemma~\ref{lemma:step:concentration}, for any large enough constant $r$, we have
    $$
      \Pr\Big(\|\alpha(X_{t+1})-\kappa(F(\alpha_1(X_{t})))\|_{\infty} \ge \frac{(1-\eta)r}{\sqrt{n}}\Big)\le O(r^{-2}).
     $$
     The same holds for $Y_{t+1}$. 

    Similarly, by assumption, $d_t^{(k)} = |\alpha_1(X_t) - \alpha_{1,k}(X_t)| \le r /\sqrt{n}$.
    Then, using Lemma~\ref{lemma:step:concentration_basket}, for any large enough constant $r$, we have that there exists $\hat\eta \in (0,1)$ such that
    $$
        \Pr\Big(d^{(k)}_{t+1} \ge \frac{r}{\sqrt{n}}\Big) \le
      \Pr\Big(|d^{(k)}_{t+1} - \hat \eta d^{(k)}_{t}| \ge \frac{(1-\hat\eta)r}{\sqrt{n}}\Big) \le O(r^{-2}).
     $$
     The same holds for all $k \in [q]$ by a union bound, and for $Y_{t+1}$. The result then follows by another union bound over the four events.
\end{proof}

\begin{proof}[Proof of Lemma~\ref{lemma:boosting}]
    We consider the coupling that in each step aims to couple two instances of the projection chain using the coupling from Lemma~\ref{lem:basket_coupling}. If it succeeds, the two instances remain coupled in all future steps.
    For ease of notation, let $\mathcal P$ be the set of pairs of configurations of the 
    projection chain 
    that satisfy A1 and A2 for a constant $\hat C$ sufficiently large,
    let $\mathcal S_{t}$ be the event that $A_{t}=A'_{t}$, and let $Z_t = (A_t,A_t')$. 
    Let $\hat{\mathcal P} \subseteq \mathcal P$ be the subset of $\mathcal P$ that contains pairs 
    of distinct configurations:
    i.e., the pairs $(A,A') \in \mathcal P$ such that $A \neq A'$.
    
    It suffices to show that for a suitable constant $\ell$, we have $\Pr\big(\cap_{i=1}^\ell \neg \mathcal S_{t+i}\big) < 1/4$. For this, we obtain first the following recurrence:
    \begin{align}
        &\Pr\big(\cap_{i=1}^\ell \neg \mathcal  S_{t+i}\big) = \Pr\big(\neg \mathcal  S_{t+\ell} \mid \cap_{i=1}^{\ell-1} \neg \mathcal  S_{t+i}\big)\Pr\big(\cap_{i=1}^{\ell-1} \neg \mathcal  S_{t+i}\big) \notag \\
        &\le \sum_{Z \in \mathcal P}\Pr\big(\neg \mathcal S_{t+\ell} , Z_{t+\ell-1}=Z \mid \cap_{i=1}^{\ell-1} \neg \mathcal S_{t+i}\big)\Pr\big(\cap_{i=1}^{\ell-1} \neg \mathcal  S_{t+i}\big) + \Pr(Z_{t+\ell-1} \not\in \mathcal P) \notag\\
        &= \sum_{Z \in \hat{\mathcal P}}\Pr\big(\neg \mathcal S_{t+\ell} , Z_{t+\ell-1}=Z \mid \cap_{i=1}^{\ell-1} \neg \mathcal S_{t+i}\big)\Pr\big(\cap_{i=1}^{\ell-1} \neg \mathcal S_{t+i}\big) + \Pr(Z_{t+\ell-1} \not\in \mathcal P),   \label{eq:final}    
        \end{align}
where the first inequality follows from the law of total probability, and 
    the second equality from the fact that under our coupling the probability of $\neg \mathcal S_{t+\ell}$ is $0$ when $Z \in \mathcal P \setminus \hat{\mathcal P}$.
The right-hand-side of~\eqref{eq:final} satisfies:           \begin{align*}  
         &= \sum_{Z \in \hat{\mathcal P}}\Pr\big(\neg \mathcal S_{t+\ell} \mid \cap_{i=1}^{\ell-1} \neg \mathcal S_{t+i},Z_{t+\ell-1}=Z\big)\Pr(Z_{t+\ell-1}=Z \mid \cap_{i=1}^{\ell-1} \neg \mathcal S_{t+i})\Pr\big(\cap_{i=1}^{\ell-1} \neg \mathcal S_{t+i}\big) + \Pr(Z_{t+\ell-1} \not\in \mathcal P) \\
        &= \sum_{Z \in \hat{\mathcal P}}\Pr\big(\neg \mathcal S_{t+\ell} \mid Z_{t+\ell-1}=Z\big)\Pr(Z_{t+\ell-1}=Z \mid \cap_{i=1}^{\ell-1} \neg \mathcal S_{t+i})\Pr\big(\cap_{i=1}^{\ell-1} \neg \mathcal S_{t+i}\big) + \Pr(Z_{t+\ell-1} \not\in \mathcal P) \\
        &\le (1-\delta)\sum_{Z \in \hat{\mathcal P}}\Pr(Z_{t+\ell-1}=Z \mid \cap_{i=1}^{\ell-1} \neg \mathcal S_{t+i})\Pr\big(\cap_{i=1}^{\ell-1} \neg \mathcal S_{t+i}\big) + \Pr(Z_{t+\ell-1} \not\in \mathcal P) \\
        &\le (1-\delta)\Pr\big(\cap_{i=1}^{\ell-1} \neg \mathcal S_{t+i}\big) + \Pr(Z_{t+\ell-1} \not\in \mathcal P),
    \end{align*}
    where the second equality follows from the Markov property; in particular, note that $\{Z_t\}$ is a Markov chain and $Z \in \hat{\mathcal P}$ is such that $\neg \mathcal S_{t+\ell-1}$ holds and can be dropped from the conditioning. The first inequality follows from Lemma~\ref{lem:basket_coupling}.
    
    Observe also that for $j \in \{1,\dots ,\ell-1\}$, it follows from Fact~\ref{fact:step:preserv} that there is a constant $\delta_2 >0$ that we can choose arbitrarily close to 1 by taking $\hat C$ large enough so that
    \begin{align*}
    \Pr(Z_{t+j} \in \mathcal P) &\ge \sum_{Z \in \mathcal P} \Pr(Z_{t+j} \in \mathcal P \mid Z_{t+j-1}=Z )\Pr( Z_{t+j-1}=Z)\\
    &\ge \sum_{Z \in \mathcal P} \delta_2 \Pr( Z_{t+j-1}=Z) \ge \delta_2 \Pr(Z_{t+j-1} \in \mathcal P),
    \end{align*}
    and iterating we get 
    $
    \Pr(Z_{t+j} \in \mathcal P) \ge \delta_2^j, 
    $
    since by assumption $Z_t \in \mathcal P$.
    Then,
     \begin{align*}
    \Pr\big(\cap_{i=1}^\ell \neg \mathcal S_{t+i}\big)
    &\le (1-\delta)\Pr\big(\cap_{i=1}^{\ell-1} \neg \mathcal S_{t+i}\big) + 1-\delta_2^{\ell-1}.
    \end{align*}    
    Iterating once again,
     \begin{align*}
    \Pr\big(\cap_{i=1}^\ell \neg \mathcal  S_{t+i}\big)
    &\le (1-\delta)^\ell + \sum_{j=1}^{\ell-1} (1-\delta_2^{j})(1-\delta)^{\ell-j-1} \le (1-\delta)^\ell + (1-\delta)^\ell \sum_{j=1}^{\ell-1} \frac{1-\delta_2^{j}}{(1-\delta)^{j+1}}.
    \end{align*}    
    Now first fixing $\ell$ so that $(1-\delta)^\ell < 1/8$, we can then pick $\delta_2$ close enough to $1$ so that $\sum_{j=1}^{\ell-1} \frac{1-\delta_2^{j}}{(1-\delta)^{j+1}} < 1$, and the result follows.
\end{proof}

\section*{Acknowledgments}

This work transpired from a discussion between the first author, Andreas Galanis, Reza Gheissari, Eric Vigoda, and Daniel {\v{S}}tefankovi\v{c}  during the Workshop on Algorithms and Randomness at Georgia Tech in 2018 in which the high-level idea of the proof was outlined. 
    
\bibliographystyle{alpha}
\bibliography{references}

\appendix

\section{Proof of facts about the drift function $F$}
\label{appendix:f}

We include in this appendix the proofs of Facts~\ref{fact:f:d2} and~\ref{fact:f:d:exp} from Section~\ref{subsec:prelim}.

\begin{proof}[Proof of Fact~\ref{fact:f:d2}]
    From direct calculation, it can be checked that
    \[
        F''(x)
        =-\frac{(q-1)\,\beta\,\theta(\beta x)\,e^{-x\beta\,\theta(\beta x)}
        \bigl(x\beta\bigl(\theta(\beta x)+2e^{-x\beta\,\theta(\beta x)}\bigr)-2\bigr)}
        {q\bigl(1 - x\beta\,e^{-x\beta\,\theta(\beta x)}\bigr)^3}.
    \]
    From \eqref{eq:rg:root}, we have
    $
        x\beta = -\frac{\ln\bigl(1 - \theta(\beta x)\bigr)}{\theta(\beta x)};
    $
    letting let $y=\theta(\beta x)$, we obtain the analogous function: 
    \[
        G(y) = -\frac{(q - 1)\beta y(1 - y)\left(-\frac{\ln(1 - y)}{y}(2 - y) - 2\right)}{q\left(1 + \frac{\ln(1 - y)}{y}(1 - y)\right)^3}.
    \]
    (Note that when $x\in(0,1]$, we have$y\in(0,1)$.) Taylor expanding $G$ about $0$, we obtain that $G(y)=-\frac{4(q-1)\beta}{3q}+O(y)$,
    and so $\lim_{y \rightarrow 0} G(y) = -\frac{4(q-1)\beta}{3q}$. Moreover,
    $\lim_{y \rightarrow 1^-} G(y) = 0$.
    Since $G$ in continuous in $(0,1)$, and its limit at $0$ and $1$ are bounded, then $G$ must be bounded in $[0,1]$, as claimed.
\end{proof}

\begin{proof}[Proof of Fact~\ref{fact:f:d:exp}]
    From direct calculation of the derivative of $F$, we deduce from~\eqref{eq:rg:root} that
    \[
        F'(x)=\frac{(q-1)\theta(\beta x)}{q(1-x\beta e^{-x\beta \theta(\beta x)})} = \frac{q-1}{q} \cdot \frac{\theta(\beta x)^2}{\theta(\beta x)+(1-\theta(\beta x))\log(1-\theta(\beta x))}.
    \]
    Then,
    $F'(a)\ge\frac{q-1}q$ when
    $
        \frac{\theta(\beta x)^2}{\theta(\beta x)+(1-\theta(\beta x))\log(1-\theta(\beta x))}\ge 1
    $
    which follows from the fact that 
    $\log(1-\theta(\beta x)) \le -\theta(\beta x)$
    since $\theta(\beta x)\in(0,1)$.
\end{proof}

\section{Proof of Lemma~\ref{lemma:potts:conc}}
\label{app:potts:conc}

For completeness, we provide in this appendix a proof of Lemma~\ref{lemma:potts:conc}. 
Let $\pi$ be the random-cluster measure on the complete graph $G=(V,E)$ on $n$ vertices with parameters $q$ and $p = 1 - e^{-\beta}$. That is, for $A \subseteq E$, let
$$
\pi(A) =  \frac{1}{Z_{\textsc{rc}}} p^{|A|} (1-p)^{|E\setminus A|}q^{c(V,A)},
$$
where $c(V,A)$ is the number of connected components in the subgraph $(V,A)$ and $Z_{\textsc{rc}}$ is the corresponding partition function.
For $A \subseteq E$,
let $L_i(A)$ denote the size of the $i$-th largest component of $A$ (breaking ties arbitrarily) and let 
$\theta = \frac{aq-1}{q-1}$.

We will require the following result for the random-cluster measure, which can be derived straightforwardly from known results in the random graph and random-cluster literature.

\begin{lemma}
\label{lemma:rc}
    Let $A \sim \pi$. Then, there exist constants $\gamma_0,\gamma_1 > 0$  such that $|L_1(A)-\theta n| \le \gamma_0 \sqrt{n}$ and $\sum_{i\ge 2} L_i(A)^2 \le \gamma_1 n$ with probability at least $0.9$.
\end{lemma}
\begin{proof}
    By Lemma 3.2 in~\cite{ABthesis}, $|L_1(A)-\theta n|  = O(n^{8/9})$ with probability $1-O(n^{-1})$.
    Moreover, Lemma 7 in~\cite{cooper2000mixing}
    and Lemma 3.1 in~\cite{BGJ} imply that $L_2(A) = O(\log n)$ with probability $1-O(n^{-1})$.
    Let $\mathcal B$ be the event that $|L_1(A)-\theta n|  = O(n^{8/9})$ and $L_2(A) = O(\log n)$ and observe that by a union bound $\Pr(\mathcal B) \ge 1 - O(n^{-1})$. 
    
    Let $M = \sum_{i\ge 2} L_i(A)^2$.
    Consider the process of coloring the components of $A$ red each with probability $1/q$ independently. Let $R$ correspond to the number of red vertices (excluding the largest component of $A$) and $\hat M = \sum_{i\ge 2} L_i(A)^2 \xi_i$ where the $\xi_i$'s are independent $\mathrm{Bernoulli}(1/q)$ random variables. 
    Observe that $\E[M]/q = \E[\hat M]$. We derive a bound for $\E[\hat M]$ and obtain the probabilistic bound for $M$ via Markov's inequality. Let $J = [\frac{1-\theta}{q}n - \delta n^{8/9}, \frac{1-\theta}{q}n + \delta n^{8/9}]$ for a suitably large constant $\delta > 0$. Then,
    \begin{align*}
        \E[\hat M] &\le n^2 \Pr(R \not\in J) + \sum_{r \in J} \E[\hat M \mid R=r] \Pr(R=r) \\
        &\le n^2(\Pr(R \not\in J \mid \mathcal B) + \Pr(\neg \mathcal B)) + \sum_{r \in J} O(r) \Pr(R=r)  \\
        &\le O(n) + \E[R] = O(n),
    \end{align*}
    where in the second inequality we used Lemma 3.1 from~\cite{BGJ} and Lemma~\ref{lemma:rg:sucep}, 
    and in third that by Hoeffding's inequality $\Pr(R \not\in J \mid \mathcal B) = O(n^{-2})$. Then, $\E[M] = O(n)$ and  by Markov's inequality there exists a large enough constant $\gamma_1 > 0$ such that $\Pr(M \le \gamma_1 n) \ge 0.95$.

    Finally, the local limit theorem in Theorem 14 from~\cite{LL} implies that for a suitable constant $\gamma_0 > 0$, we have
    $|L_1(A)-\theta n| \le \gamma_0 \sqrt{n}$
    with probability  at least $0.95$; the result then follows by a union bound.
\end{proof}

We now use this lemma to prove Lemma~\ref{lemma:potts:conc} using the well-known connection between the random-cluster and the ferromagnetic Potts distributions. In particular, if $A \sim \pi$ and we perform a coloring step of the SW dynamics, the resulting configuration has distribution $\mu$~\cite{ES}.

\begin{proof}[Proof of Lemma~\ref{lemma:potts:conc}]  
    Let $A \sim \pi$ and obtain $X \in \Omega$ by 
    assigning spins uniformly at random from $\{1,\dots,q\}$ to each connected component of $A$ independently; then, $X$ has distribution $\mu$. 
   Let $\mathcal{A}$ be the event hat $|L_1(A) - \theta n| \le \gamma_0\sqrt{n}$ and $\sum_{i\ge 2} L_i(A)^2 \le \gamma_1 n$.
   The largest spin class will be
   the one assigned to the largest component of $A$ and will have expected size $\theta + \frac{1-\theta}{q} = a$; all other colors classes will have expected size  $\frac{1-\theta}{q} = \frac{1-a}{q-1}$. Under the assumption that $\sum_{i\ge 2} L_i(A)^2 \le \gamma_1 n$, we can prove the concentration of $\alpha(X)$ about $m$  via Hoeffding's inequality. In particular, for $\ell > 0$: 
    $$
    \Pr\Big(\|\alpha(X) - m\|_\infty \ge 
    \ell \sqrt{\frac{\gamma_1}{n}} \,\Big|\, \mathcal A
     \Big) \le 2 e^{-2\ell^2}, 
    $$
    and for a suitable $\ell >0$ we can bound this probability by $0.1$. Since by Lemma~\ref{lemma:rc}, $\Pr(\mathcal A) \ge 0.9$, we obtain that
    $$
    \Pr\Big(\|\alpha(X) - m\|_\infty \le    \ell \sqrt{\frac{\gamma_1}{n}}  \Big) \ge 3/4
    $$
    as desired.
\end{proof}

\end{document}